\documentclass{elsart3-1}

\usepackage{amssymb}

\usepackage[english,francais]{babel}

\newenvironment{proof}{\paragraph{Proof:}}{\hfill$\square$}
\newtheorem{theorem}{Theorem}[section]
\newtheorem{lemma}[theorem]{Lemma}
\newtheorem{e-proposition}[theorem]{Proposition}
\newtheorem{corollary}[theorem]{Corollary}
\newtheorem{e-definition}[theorem]{Definition\rm}

\newtheorem{theoreme}{Th\'eor\`eme}[section]

\newtheorem{proposition}[theoreme]{Proposition}

\renewcommand{\theequation}{\arabic{equation}}
\newcommand{\er}{\operatorname{er}}

\def\og{\leavevmode\raise.3ex\hbox{$\scriptscriptstyle\langle\!\langle$~}}
\def\fg{\leavevmode\raise.3ex\hbox{~$\!\scriptscriptstyle\,\rangle\!\rangle$}}

\journal{Arxiv}
\begin{document}
\centerline{}
\begin{frontmatter}


\selectlanguage{english}
\title{On cogrowth function of uniformly recurrent sequences}


\selectlanguage{english}
\author[authorlabel1]{Igor Melnikov}
\ead{melnikov\_ig@mail.ru}
\author[authorlabel2]{Ivan Mitrofanov}
\ead{phortim@yandex.ru}

\address[authorlabel1]{Moscow Institute of Physics and Technology, Dolgoprudny, Russia}
\address[authorlabel2]{C.N.R.S., \'{E}cole Normale Superieur, PSL Research University, France}


\begin{abstract}
\selectlanguage{english}
For a sequence $W$ we count the number $O_W(n)$ of minimal forbidden words no longer then $n$ and prove that
 $$\overline{\lim_{n \to \infty}} \frac{O_W(n)}{\log_3n} \geq 1.$$
 \footnote{The paper was supported  by Russian Science Foundation (grant no. 17-11-01377). The second named author also received funding from the European Research Council (ERC) under
 	the European Union's Horizon 2020 
 	research and innovation program (grant agreement
 	No.725773).}

\end{abstract}
\end{frontmatter}


\selectlanguage{english}

\section{Introduction}

A language (or a subshift) can be defined by the list of {\it forbidden subwords}.
The linearly equivalence class of the counting function for minimal forbidden words is an topological invariant of the corresponding symbolic dynamical system \cite{Beal}.

G.~Chelnokov, P.~Lavrov and I.~Bogdanov \cite{BogdCheln}, \cite{Cheln1}, \cite{Lavr1}, \cite{Lavr2} estimated the minimum number of forbidden words that define a periodic sequence with a given length of period.

We investigate a similar question for uniformly recurrent sequences and prove a logarithmic estimation for {\it the cogrowth function}.

\section{Preliminaries}
An {\it alphabet} $A$ is a finite set of elements,
{\it letters} are the elements of an alphabet. 
The finite sequence of letters of $A$ is called a {\it finite word} (or {\it a word}).
An {\it infinite word}, or {\it sequence} is a map $\mathbb{N} \to A$. 

The {\it length} of a finite word $u$ is the number $|u|$ of letters in it.
The {\it concatenation} of two words $u_1$ and $u_2$ is denoted by
$u_1u_2$.

A word $v$ is a {\it subword} of a word $u$ if $u = v_1vv_2 $ for some words $ v_1 $, $ v_2 $. 
If $ v_1 $ or $ v_2 $ is an empty word, then $ v $ is 
{\it prefix} or {\it suffix} of $u$ respectively.

A sequence $W$ on a finite alphabet is called {\it periodic} if it has form $W=u^{\infty}$ for some finite word $u$.

A sequence of letters $W$ on a finite alphabet is called 
{\it uniformly recurrent} if for any finite subword $u$ of $W$ there exists a number $C(u, W)$ such that any subword of $W$ with length $C(u, W)$ contains $u$.

A finite word $u$ is called an {\it obstruction} for $W$ if it is not a subword of $W$ but any its proper subword is a subword of $W$.
The {\it cogrowth function} $O_W(n)$ is the number of   obstructions with length $\leqslant n$. 

Further we assume that the alphabet $A$ is binary, $A =\{\alpha, \beta\}$. 

The main result of this article is the following

\begin{theorem}\label{thm:ur}
	Let $W$ be an uniformly recurrent non-periodic sequence on a binary alphabet. 
	Then $$\overline{\lim_{n \to \infty}} \frac{O_W(n)}{\log_3n} \geq 1.$$
	
\end{theorem}

Note that if $F = \alpha \beta \alpha\alpha \beta\alpha \beta\alpha\alpha \beta \alpha \dots$ is the {\it Fibonacci sequence}, then 

$O_F(n) \sim \log_{\varphi}n$, where $\varphi = {\frac{\sqrt5 + 1}2}$ \cite{Beal}.

\section{Factor languages and Rauzy graphs}

A {\it factor language} $\mathcal{U} $ is a set of finite words such that for any $u\in \mathcal{U}$ all subwords of $u$ also belong to $\mathcal{U}$. 
A finite word $u$ is called an {\it obstruction} for $\mathcal{U}$ if $u\not \in \mathcal{U}$, but any its proper subword belongs to $\mathcal{U}$.

For example, the set of all finite subwords of a given 
sequence $W$ forms a factor language denoted by $\mathcal{L}(W)$.

Let $\mathcal{U}$ be a factor language and $k$ be an integer. 
The  {\it Rauzy graph}
$R_k(\mathcal{U})$ of order $k$ is the directed graph with the vertex set $\mathcal{U}_k$ and the edge set $\mathcal{U}_{k+1}$.

Two vertices $u_1$ and $u_2$ of $R_k(\mathcal{U})$ are connected by an edge $u_3$ if and only if $u_1, u_2, u_3 \in \mathcal{U}$, $u_1$ is a prefix of $u_3$, and $u_2$ is a suffix of $u_3$.

For a sequence $W$ we denote the language $R_k(\mathcal{L}(W))$ by $R_k(W)$.

Further the word {\it graph} will always mean a directed graph, the word {\it path} will always mean a {\it directed path} in a directed graph.
The {\it length} $|p|$ of a path $p$ is the number of its vertices, i.e. the number of edges plus one.

If a path $p_2$ starts at the end of a path $p_1$, we denote their concatenation by $ p_1p_2 $.
It is clear that $|p_1p_2| = |p_1| + |p_2| - 1$.

Recall that a directed graph is {\it strongly connected}  if it contains a directed path from $v_1$ to $v_2$ and a directed path from $v_2$ to $v_1$ for every pair of vertices $\{v_1,v_2\}$.

\begin{proposition}
	Let $W$ be an uniformly recurrent non-periodic sequence. 
	Then for any $k$ the graph $R_k(W)$ is strongly connected and is not a cycle.
\end{proposition}
\begin{proof}
	Let $u_1$, $u_2$ be two elements of $\mathcal{L}(W)_k$.
	Since $W$ is uniformly recurrent then ? $W$ 
	contains a subword of form $u_1uu_2$. 
	The subwords of $u_1uu_2$ of length $k+1$ form in $R_k(W)$ a path connecting $u_1$ and $u_2$.
	
	Assume that $R_k(W)$ is a cycle of length $n$. 
	Then it is clear that $W$ is periodic and $n$ is the length of its period.  
\end{proof}

If $H$ is a directed graph, its {\it directed line graph} $f(H)$
has one vertex for each edge of $H$. 
Two vertices of $f(H)$ representing directed edges $e_1$ from $v_1$ to $v_2$ and $e_2$ from $v_3$ to $v_4$ in $H$ 
are connected by an edge from $e_1$ to $e_2$ in $f(H)$ 
when $v_2 = v_3$. 
That is, each edge in the line digraph of $H$ represents a length-two directed path in $H$.

Let $\mathcal{U}$ be a factor language. 
A path $p$ of length $m$ in $R_n(\mathcal{U})$ corresponds to a word of length $n + m - 1$.

The graph $R_m (\mathcal{U}) $ can be considered 
as a subgraph of $f^{m-n}(R_n(\mathcal{U}))$.
Moreover, the graph $R_{n+1}(\mathcal{U})$ is obtained from $f(R_{n}(\mathcal{U}))$ by deleting edges that correspond to obstructions of $\mathcal{U}$ of length $n+1$.

We call a vertice $v$ of a directed graph $H$ {\it a fork} if $v$ has out-degree more than one.
Further we assume that all forks have out-degrees exactly 2 (this is the case of a binary alphabet).

For a directed graph $H$ we define its {\it entropy regulator}: $er(H)$ is the minimal integer such that any directed path of length $er(H)$ in $H$ contains at least one vertex that is a fork in $H$. 

Now we prove some facts about  entropy regulators.

\begin{proposition}
	Let $H$ be strongly connected digraph that is not a cycle, then $\er(H) < \infty$.
\end{proposition}

\begin{proof}
	
	Assume the contrary. 
	Let $n$ be the total number of vertices in $H$. 
	Consider a path of length $n + 1$ in $ H $ that does not contain forks. 
	Note that this path visits some vertex $ v $ at least twice.
	This means that starting from $v$ it is possible to obtain only vertices of this cycle.
	Since the graph $H$ is strongly connected, $H$ coincides with this cycle.
\end{proof}

\begin{lemma}\label{le:del_edge}
	Let $H$ be a strongly connected digraph, $\er(H) = L$, 
	let $v$ be a fork in $H$, the edge $e$ starts at $v$.
	
	Let the digraph $H^*$ be obtained from $H$ by removing the edge $e$.
	Let $G$ be a subgraph of $H^*$ that consists of all vertices and edges reachable from $v$.
	Then $G$ is strongly connected digraph.
	Also $G$ is either a cycle of length at most $L$, or $\er(G) \leq 2L $.
\end{lemma}

\begin{proof}
	
	First we prove the digraph $G$ is strongly connected.
	Let $ v '$ be an arbitrary vertex of $G$, then there is a path in $G$ from 
	$ v $ to $ v' $.
	Consider a path $p$ of minimum length from $ v '$ to $ v $ in $H$. Such path exists, otherwise $H$ is not strongly connected.
	The path $p$ does not contain the edge $ e $, otherwise it could be shortened.
	This means that $p$ connects $ v '$ with $ v $ in the digraph $ G $.
	From any vertex of $G$ we can reach the vertex $ v $, hence  $ G $ is strongly connected.
	
	Consider an arbitrary path $ p $ of length $ 2L $ in the digraph $ G $, suppose that $p$ does not have forks.
	Since $ \er (H) = L $, then in $ p $ there are two vertices $ v_1 $ and $ v_2 $ such that they are forks in $H$ and there are no forks in $p$ between $v_1$ and $v_2$. 
	The out-degrees of all vertices except $ v $ coincide in $ H $ and $ G $.
	If $v_1\neq v$ or $ v_2 \neq v $, then we find a vertex of $p$ that is a fork in $ G $.

	If $v_1 = v_2 = v$, then there is a cycle $C$ in $ G $ such that $|C| \leq L$ and $C$ does not contain forks of $G$.
	Since $ G $ is a strongly connected graph, it coincides with this cycle $C$.
	
\end{proof}

\begin{lemma}\label{le:evol}
	Let $H$ be a strongly connected digraph, $\er(H) = L$. 
	Then $\er(f(H)) = L$. 
\end{lemma}
\begin{proof}
	
	The forks of the digraph $ f (H) $ are 
	edges in $H$ that end at forks.
	Consider $L$ vertices forming a path in $f(H)$.
	This path corresponds to a path of length $L + 1 $ in $ H $.
	Since $\er(H) \leq L$, there exists an edge of this path that ends at a fork.
\end{proof}

\begin{corollary}\label{cor:er}
	Let $ W $ be a binary uniformly recurrent non-periodic sequence; then for any $n$ 
	$$
	\er(R_{n-1}(W)) \leq 2^{O_W(n)}.
	$$
\end{corollary}

\begin{proof}
	We prove this by induction on $n$.
	The base case $n=0$ is obvious. 
	
	Let $ \ er (R_ {n-1} (W)) = L $ and suppose $ W $ has exactly $ a $ obstructions of length $ n + 1 $.
	These obstructions correspond to paths of length 2 in the graph $R_ {n-1} (W)$, i.e. edges of the graph 
	$ H: = f (R_ {n-1} (W)) $.
	From Lemma \ref{le:evol} we have that $ \ er (H) = L $.
	
	The graph $ R_ {n} (W) $ is obtained from the graph $ H $ by removing some edges $ e_1, e_2, \dots, e_a $.
	Since $ W $ is a uniformly recurrent sequence, the digraphs $ H $ and $ H - \{e_1, e_2, \dots, e_a \} $ are strongly connected.
	This means that the edges $ e_1, \dots, e_a $ start at different forks of $ H $. 
	We also know that $ R_n (W) $ is not a cycle.
	The graph $ R_n (W) $ can be obtained by removing edges $e_i$ from $ H $ one by one. 
	Applying Lemma \ref{le:del_edge} $a$ times, we show that $ \er (R_n (W)) \leq 2^aL $, which completes the proof.
	
\end{proof}

\section{Proof of Theorem \ref{thm:ur}}

\begin{proposition}\label{pr:prol}
	
	Let $H$ be a strongly connected digraph, let $ p $ be a path in $H$, let a fork $v$ be the starting point of the last edge of $p$.
	We call a path $t$ in $H$ \emph{good} is $t$ does not contain $p$ as a sub-path. 
	Then for any good path $ s $ there exists an edge $ e $ such that $ se $ is also a good path.
	Moreover, if the last vertex of $ s $ is a fork $ v '\neq v $, then there are two such edges.
	
\end{proposition}
\begin{proof}
	
	If the last vertex of $s$ is not $v$, 
	then we can take any edge outgoing from it.
	If the last vertex of $s$ is $ v $, then 2 edges $ e_1 $ and $ e_2 $ go out of $v$.
	One of them is the last edge of the path $ p $, so we can take another edge.
	
\end{proof}

\begin{lemma}
	Let $H$ be a strongly connected digraph, $\er(H) = L$.
	Let $u$ be an arbitrary edge of the graph $f^{3L}(H)$, 
	then the digraph $f^{3L}(H) - u$ contains a strongly connected subgraph $B$ such that $\er(B)\leq 3L$.
\end{lemma}

\begin{proof}
	The edge $u$ of the graph $f^{3L} (H) $ corresponds to the path $ p_u $ in the graph $H$, $|p_u| = 3L + 2 $.
	
	This path visits at least 3 forks (taking into account the number of visits).
	Next, we consider three cases.
	
	{\bf Case 1.} 
	Assume the path $ p_u $ visits at least two different forks  of $ H $.
	Let $ v_1 $, $ v_2 $ be two different forks in $ H $,
	let $ p_ {1} e_2 $ be a sub-path of $ p_u $, where
	the path $ p_{1} $ starts at $ v_1 $ and ends at $ v_2 $ and does not contain forks other than $v_1$ and $v_2$,
	and let the edge $e_2$ go out of $ v_2 $.
	
	It is clear that the length of $ p_1 $ does not exceed $ L + 1 $.
	Lemma \ref{le:del_edge} implies that there is a strongly connected subgraph $G$ oh $H$ such  that $G$ contains the vertex $ v_2 $ but does not contain the edge $e_2$.
	
	If $ G $ is not a cycle, then $ \er (G) \leq 2L $.
	Hence, the graph $ B: = f ^ {2L} (A) $ is a subgraph of $ f ^ {2L} (H) $, and from Lemma \ref{le:evol} we have 
	$\er(B) \leq 2L $. The edges of $ B $ are paths in $ G $ and do not contain $ e_2 $, which means that $ B $ does not contain the edge $ u $.

	If $ G $ is cycle, we denote it by $ p_{2} $ (we assume that $ v_2 $ is the first and last vertex of $ p_ {2} $).
	The length of $ p_2 $ does not exceed $ L $.
	Among the vertices $ p_2 $ there are no forks of $H$ besides $ v_2 $.
	Therefore, $ v_1 \not \in p_2 $.
	
	Call a path $t$ in  $ H $ {\it good}, if $t$
	does not contain the sub-path $ p_1e_2 $.
	
	Let us show that if $ s $ is a good path in $ H $, then there are two different paths $ s_1 $ and $ s_2 $ starting at the end of $ s $ such that $ | s_1 | = | s_2 | = 3L $ and the paths $ ss_1 $, $ ss_2 $ are also good.
	
	Proposition \ref{pr:prol} says that for any good path we can add an edge an obtain a good path.
	
	There is a path $ t_1 $, $ | t_1 | <L $ such that $ st_1 $ is a good path and ends at some fork $v$.
	If $v\neq v_2 $, then two edges $ e_i $, $ e_j $ go out from $v$, the paths $ st_1e_i $ and $ st_2e_j $ are good, and each of them can be prolonged further to a good path of arbitrary length.

	If $v = v_2$, then the paths $st_1p_2p_2$ and $st_1p_2e_2$ are good.
	
	Consider in $f^{3L}(H)$ a subgraph that consists of all vertices and edges that are good paths in $ H $, let $B$ be a strongly connected component of this subgraph.
	
	We proved that $ \er (B) \leq 3L $.
	In addition, it is clear that $ B $ does not contain the edge $ u $.
	
	{\bf Case 2.} Assume that the path $p_u$  visits exactly one fork $v_1$ (at least trice), but there are forks besides $v_1$ in $H$.
	
	There are two edges $ e_1 $ and $ e_2 $ that go out from $ v_1 $.
	Starting with these edges and and moving until forks, 
	we obtain two paths $ p_1 $ and $ p_2 $.
	The edge $ e_1 $ is the first edge of $ p_1 $, the edge $ e_2 $ is the first of $ p_2 $, 
	and $ | p_1 |, | p_2 | \leq L $.
	
	Since $ p_u $ goes through $ v_1 $ more than once and does not contain other forks, one of $p_1, p_2$ is a cycle.
	
	Without loss of generality, the path $p_1$ starts and ends at 
	$ v_1 $.
	If $ p_2 $ also ends at $ v_1 $, 
	then from $ v_1 $ it is impossible to reach any other fork.
	Therefore, $ p_2 $ ends at some fork $ v_2 \neq v_1 $.
	Since $ p_u $ visits  $ v_1 $ at least three times and does not contain other forks, $p_u$ has sub-path $ p_1e_1 $.
	
	We call a path {\it good} if it does not contain $ p_1e_1 $.
	We  show that if $ s $ is a good path in $ H $, then there are two different paths $ s_1 $ and $ s_2 $ starting at the end of $ s $ such that $ | s_1 | = | s_2 | = 3L $ and the paths $ ss_1 $, $ ss_2 $ are also good.
	
	There is a path $ t_1 $ such that $ | t_1 | <L $ and the path $ st_1 $ is a good path ending at some fork $v$.
	If $v = v_1 $, take $ t_2: = t_1p_2 $, otherwise we take $ t_2 = p_2 $.
	We see that $ | t_2 | \leq 2L $, the path $ st_2 $ is good and ends at some  fork in $ v '\neq v_1 $.
	The proposition \ref{pr:prol} shows that the path $st_2$ can be prolonged to the right at least in two ways.

	We complete the proof as the previous case. 
	Consider in $f^{3L}$ a subgraph consisting of all vertices and edges corresponding to good paths in $H$ and take the strongly connected component $ B $ in this subgraph.
	
	{\bf Case 3.} 
	Assume that the path $p_u$  visits exactly one fork $v_1$ (at least trice), and there are no forks in $H$ besides $v_1$.
	
	The edges $ e_1 $ and $ e_2 $ go out from $ v_1 $, 
	the cycles $ p_1 $ and $ p_2 $ start and end at $ v_1 $ and do not contain other forks, $ e_1 $ is the first edge of $ p_1 $, 
	$ e_2 $ is the first edge of $ p_2 $, $ | p_1 |, | p_2 | \leq L $.

	The path $ p_u $ contains a sub-path of the form $p_ip_je_k $, where $ i, j, k \in \{1, 2 \} $.
	
	There are two cases.
	
	{\bf Case 3a} Assume that $i = j$ or $j=k$.
	
	Without loss of generality we assume that $ p_u $ contains a sub-path $ p_1e_1 $.
	
	We call a path {\it good} if it does not contain a sub-path 
	$ p_1e_1 $.
	We  show that if $ s $ is a good path in $ H $, then there are two different paths $ s_1 $ and $ s_2 $ starting at the end of $ s $ such that $ | s_1 | = | s_2 | = 3L $ and the paths $ ss_1 $, $ ss_2 $ are also good.

	First we take a path $t_1$ such that $|t_1|< L$ and $st_1$ is a good path ending at $v_1$.
	The paths $st_1p_2e_2$ and $st_1p_2e_1$ are good.
	
	We complete the proof as in the previous cases.
	
	{\bf Case 3b} Assume that $i \neq j$ and $j\neq k$. 
	Without loss of generality we assume that $ p_u $ contains a sub-path $p_1p_2e_1$.
	
	We call a path in $H$ {\it good} if it does not contain a sub-path 
	$p_1p_2e_1$.
	We  show that if $ s $ is a good path in $ H $, then there are two different paths $ s_1 $ and $ s_2 $ starting at the end of $ s $ such that $ | s_1 | = | s_2 | = 3L $ and the paths $ ss_1 $, $ ss_2 $ are also good.
	
	First we take a path $t_1$ such that $|t_1|< L$ and $st_1$ is a good path ending at $v_1$.
	The paths $st_1p_2p_2e_1$ and $st_1p_2p_2e_2$ are good. 
	Note that
	$|t_1p_2p_2e_1|= |t_1p_2p_2e_2| \leq 3L$.
	
	Again, we complete the proof as in the previous cases.
	
\end{proof}

\begin{corollary}\label{cor:main}
	
	Let $H$ be a strongly connected digraph, $\er(H) = L$, $k \geq 3L$.
	Let $u$ be an arbitrary edge of the graph $f^{k}(H)$, 
	then the digraph $f^{k}(H) - u$ contains a strongly connected subgraph $B$ such that $\er(B)\leq 3L$.
	
\end{corollary}

\begin{lemma}\label{le:subseq}
	Let $a_n$ be a sequence of positive numbers such that  
	$$
	\underline{\lim}_{k\to \infty} \frac{\log_3 a_k}{k} > 1.
	$$
	
	Then there exists $n_0$ such that for any $k>0$
	$$
	a_{n_0 + k} - a_{n_0} > 4 \cdot 2^{n_0}\cdot 3^k
	$$ 
\end{lemma}

\begin{proof}
	Let us denote $a_k/3^k$ by $b_k$.
	It is clear that  $\lim_{k\to \infty}b_k = \infty$.
	Hence, there exists $n_0$ such that  ??? $b_{n_0} > 10$ and $b_{n} \geq b_{n_0}$ for all $n > n_0$.
	Then for any $k > 0$ it holds
	$$
	a_{n_0 + k} - a_{n_0} = b_{n_0+k}3^{n_0 + k} - b_{n_0}3^{n_0} \geq b_{n_0}3^{n_0}(3^k - 1) > 4 \cdot 2^{n_0}\cdot 3^k.
	$$
\end{proof}

Now we are ready to prove Theorem \ref{thm:ur}.

\begin{proof}
	
	Arrange all the obstructions of the uniformly recurrent binary sequence $ W $ by their lengths:
	$$
	|u_1| \leq |u_2| \leq \dots 
	$$
	
	If $\underline{\lim}_{k\to \infty}\frac{\log_3 |u_k|}{k} \leq 1$, then the statement of the theorem holds.
	
	Assume the contrary. Lemma \ref{le:subseq} says that there is $n_0$ such that for any positive integer $k$ it holds 
	$$
	|u_{n_0 + k}| > |u_{n_0}| + 4 \cdot 2^{n_0} \cdot 3^k.
	$$
	
	For $n > n_0$ denote the number $|u_{n_0}| + 4 \cdot 2^{n_0} \cdot 3^{n-n_0}$ by $b_n$. 
	Let $b_n = |u_{n_0}|$.
	
	For all $n > n_0$ take a proper subword $v_n$ of $u_n$ such that $|v_n| = b_n$.
	
	Denote by $ \mathcal {U} $  the set of all finite binary words that do not contain as subwords the words $ u_i $ for $ 1 \leq i \leq n_0 $ and $ v_i $ for $ i> n_0 $ .
	
	We get a contradiction with the uniform recurrence of $ W $ if we show that the language $ \mathcal {U} $ is infinite.
	
	It is clear that the Rauzy graphs $R_{u_{n_0}-1}(\mathcal{U}) = R_{u_{n_0}-1}(W)$, and from Corollary \ref{cor:er} we have
	$$
	\er(R_{u_{n_0}-1}(\mathcal{L})) \leq 2^{n_0}.
	$$
	
	By induction on $ k $, we show that the graph $ R_{b_{n_0 + k} -1} (\mathcal {U}) $ contains a strongly connected subgraph $ H_k $ such that $ \er (H_{k}) \leq 3^k \cdot 2^{n_0} $.
	
	We already have the base case $k=0$. 
	
	The graph $R_{b_{n_0+k+1}-1}(\mathcal{U})$ contains a sub-graph $f^{b_{n_0+k+1}-b_{n_0+k}}(H_k)$ without at most one edge (that corresponds to the word  $v_{n_0+k+1}$).
	Note that
	$$b_{n_0+k+1}-b_{n_0+k} > 3 \cdot \er(H_k).
	$$
	
	hence we can apply Corollary \ref{cor:main}. 
	Then the digraph $R_{b_{n_0+k+1}-1}(\mathcal{U})$ has a strongly connected subgraph with entropy regulator at most $3^{k+1}\cdot 2^{n_0}$.
	
	We show that all the graphs $R_{b_n} (\mathcal {U}) $ are nonempty and, therefore, the language $ \mathcal {U} $ is infinite.
	
	On the other hand, all elements in $ \mathcal {U} $ are subwords of $ W $ and do not contain $v_{n_0 + 1}$.
	But the word $ v_ {n_0 + 1} $ is a proper subword of the obstruction $u_{n_0 + 1}$, and, therefore, $ v_ {n_0 + 1} $ is a subword of $ W $.
	This means that the infinity of the language $ \mathcal{U}$ contradicts the uniform recurrence of $ W $.
	
\end{proof}

\noindent {\bf Example.} Consider the sequence
$$
F = 0100101001001\dots
$$

defined recursively as follows: 
$u_0 = 0$, $u_1 = 01$, $u_{k}$ is the concatenation $u_{k-1}u_{k-2}$ for $k \geq 2$.
Since $u_i$ is a prefix of $u_{i+1}$, the sequence $(u_i)$ has a limit, called a {\it Fibonacci word} 

In Example 25 of \cite{Beal} the set $\{11,000,10101,00100100, \dots\}$ of minimal forbidden words of $F$ is described. These words have lengths equal to Fibonacci numbers. It is well known that the Fibonacci sequence is uniformly recurrent, hence in Theorem \ref{thm:ur} we cannot replace the constant $3$ by a number smaller than ${\frac{\sqrt5 + 1}2}$.

\end{document}